\documentclass[11pt,reqno]{amsart}
\usepackage{tikz}
\usepackage{bm}
\usepackage{subcaption}

\usepackage{hyperref}
\usepackage{epsfig}
\usepackage{math}
\usepackage[norelsize,boxed,linesnumbered,vlined,algo2e]{algorithm2e} 
\graphicspath{{}}
\usepackage{tikz}
\usetikzlibrary{shapes,arrows}
\usepackage{algorithm}
\usepackage{algorithmic}

\usepackage{media9}
\usepackage{graphicx}

\tikzstyle{decision} = [diamond, draw, fill=blue!20, 
    text width=4.5em, text badly centered, node distance=3cm, inner sep=0pt]
\tikzstyle{block} = [rectangle, draw, fill=blue!20, 
    text width=5em, text centered, rounded corners, minimum height=4em]
\tikzstyle{line} = [draw, -latex']
\tikzstyle{cloud} = [draw, ellipse,fill=red!20, node distance=3cm,
    minimum height=2em]
\usetikzlibrary{positioning}
\tikzset{main node/.style={circle,fill=blue!20,draw,minimum size=1cm,inner sep=0pt},  }

\newcommand{\wc}[1]{{\color{red} [Li: #1]}}

\usepackage[revision]{revdiff}

\title{Controlling Propagation of epidemics via mean-field controls}
\author[Lee]{Wonjun Lee}
\author[Liu]{Siting Liu} 
\author[Tembine]{Hamidou Tembine} 
\author[Li]{Wuchen Li} 
\author[Osher]{Stanley Osher}
\thanks{This work is supported by AFOSR MURI FA9550-18-1-0502.}
\begin{document}
\maketitle
\begin{abstract} { The coronavirus disease 2019 (COVID-19) pandemic is changing and impacting lives on a global scale. In this paper, we introduce a mean-field control model in controlling the propagation of epidemics on a spatial domain. The control variable, the spatial velocity, is first introduced for the classical disease models, such as the SIR model. For this proposed model, we provide fast numerical algorithms based on proximal primal-dual methods. Numerical experiments demonstrate that the proposed model illustrates how to separate infected patients in a spatial domain effectively.} 
\end{abstract}

\section{Introduction}
\label{sec:intro}
The outbreak of COVID-19 epidemic has resulted in over millions of confirmed cases and hundred thousands of deaths globally. It has a huge impact on global economy as well as everyone's daily life. There has been a lot of interest in modeling the dynamics and propagation of the epidemic. {One of the well-known and basic models in epidemiology} is the SIR model proposed by Kermack and McKendrick \cite{kermack1927contribution} in 1927. { Here, S, I, R represent the number of susceptible, infected and recovered people respectively.} They use an ODE system to describe the transmission dynamics of infectious diseases among the population. As the propagation of COVID-19 has significant spatial characteristic, actions such as travel restrictions, { physical distancing} and self-quarantine are taken to slow down the spread of the epidemic. It is important to have a spatial-type SIR model to study the spread of the infectious disease and movement of individuals \cite{kendall1965mathematical,kallen1984thresholds,hosono1995traveling}. 

Since the epidemic has affected the society and individuals significantly, mean-field games and mean field controls (MFG, MFC) provide a perspective to study and understand the underlying population dynamics. Mean-field games were introduced by Jovanovic and Rosenthal \cite{Jovanovic88},  Huang, Malham\'{e}, and Caines \cite{HCM06}, and Lasry and Lions \cite{LasryLions06a,LasryLions06b}. They model a huge population of agents playing dynamic games. 
There is growing research interest in this direction. For a review of MFG theory, we refer to \cite{lasry2007mean,gomes2014mean}. With wide application to various fields \cite{Gomes:2015th,burger2013mean,lachapelle2016efficiency,achdou19}, computational methods are also designed to solve related high dimensional MFG problems \cite{bencar'15, silva18, EHanLi2018_meanfield,lin2020apacnet, ruthotto2019machine, liu2020computational}.

In this paper, we combine the above ideas of spatial SIR model and MFG. In other words, we introduce a mean-field game (control) model for controlling the virus spreading within a spatial domain. Here the goal is to minimize the number of infectious agents and the amount of movement of the population. In short, we formalize the following constrained optimization problem 
 \begin{equation*}
    \begin{split}
        &{ \inf_{(\rho_i, v_i)_{i\in \{S,I,R\}}}} \indent E(\rho_I(T,\cdot)) + \int^T_0  \int_\Omega \sum_{i\in \{S,I,R\}} \frac{\alpha_i}{2}\rho_i |v_i|^2 + \frac{c}{2}(\rho_S + \rho_I + \rho_R)^2 dx dt
    \end{split}
\end{equation*}
subject to 
\begin{equation*}
  \left\{ \begin{aligned}
       & \partial_t \rho_S + \nabla \cdot (\rho_S v_S) + \beta \rho_S \rho_I - \frac{\eta^2_S}{2} \Delta \rho_S= 0\\
        &\partial_t \rho_I + \nabla \cdot (\rho_I v_I) - \beta \rho_S \rho_I + \gamma \rho_I - \frac{\eta^2_I}{2} \Delta \rho_I = 0\\
        &\partial_t \rho_R + \nabla \cdot (\rho_R v_R) - \gamma \rho_I - \frac{\eta^2_R}{2} \Delta \rho_R= 0\\
&\rho_S(0,\cdot), \rho_I(0,\cdot), \rho_R(0,\cdot) \text{ are given.}
    \end{aligned}\right.
\end{equation*}
{Here $\rho_i$ represents population density  and $v_i$ describes the movement, with $i\in \{S,I,R\}$  corresponding to the susceptible, infected and recovered compartmental state or class.} We consider the spatial SIR model with nonlocal spreading modeled by an integration kernel $K$ representing the physical distancing and a spatial diffusion of population, and set it as dynamic to our mean-field control problem, which is the constraint to the minimization problem.
The minimization objective include both the movement and the congestion of the population. The kinetic energy terms describes the situation that, if population (the susceptible, infected or recovered) needs to be moved to alleviate local medical shortage, there is a cost behind it. The congestion term models the fact that government don't want the population gets too concentrated in one place. This might increase the risk of disease outbreaks and their faster and wider spread.
{Due to the multiplicative nature of the interaction term between susceptible and infectious agents $\beta \rho_S \rho_I$, the mean-field control problem is  a non-convex problem.} With Lagrange multipliers, we formalize the mean-field control problem as an unconstrained optimization problem. Fast numerical algorithms are designed to solve the non-convex optimization problem in $2D$ with $G-prox$ preconditioning \cite{JacobsLegerLiOsher2018_solvinga}.

In the literature, spatial SIR models in the form of a nonlinear integro-differential \cite{aronson1977asymptotic,diekmann1979run,thieme1977model} and reaction-diffusion system \cite{kallen1984thresholds,hosono1995traveling} have been studied. Traveling waves are studied to understand the propagation of various type of epidemics, such as Lyme disease, measles etc, and recently, COVID-19 \cite{caraco2002stage, grenfell2001travelling, wang2010travelling,berestycki2020propagation}.
In \cite{berestycki2020propagation}, they introduce a SIRT model to study the effects of the presence of a road on the spatial propagation of the epidemic.
For surveys, see \cite{murray2001mathematical, ruan2007spatial}.
As for numerical modelling of epidemic model concerning spatial effect, finite-difference methods are used to discretize the reaction-diffusion system and solve the spatial SIR model and its various extensions \cite{chinviriyasit2010numerical, jaichuang2014numerical, farago2016qualitatively}.
Epidemic models have been treated using optimal control theory, with major control measures on medicare (vaccination) \cite{sethi1978optimal, lahrouz2018dynamics,jang2020optimal}.
In \cite{jang2020optimal}, a feedback control problem of SIR model is studied to help determine the vaccine policy, with the goal to minimize the number of infected people.
In \cite{li2019dynamic}, they introduce a nonlinear SIQS epidemic model on complex networks and study the optimal  quarantine control. Compared to previous works, our model is the first to consider an optimal control problem for SIR model on a spatial domain, combining optimal transport and mean field controls. As SIR model can be interpreted in terms of stochastic processes of agent-based models, it can be obtained  as a motion of the law of a three-state Markov chain with the transition from $S$ to $I$ and $I$ to $R$.\cite{allen2017primer}
Here we formulate velocity fields among S, I, R populations as control variables. And our model applies a pair of PDEs, consisting Fokker-Planck equation and Hamilton-Jacobi equation. These equations describe how different populations (susceptible, infected or recovered) react to the propagation of pandemic on a spatial domain. 

Our paper is organized as follows. In section \ref{section2}, we introduce the mean field control model for propagation of epidemics.  
We introduce a primal-dual hybrid gradient algorithm for this model in section \ref{section3}. In section \ref{section4}, several numerical examples are demonstrated. 

\section{Model}\label{section2}
In this section, we briefly review the classical epidemics models, e.g. SIR dynamics. We then introduce a mean field control model for SIR dynamics on a spatial domain. We derive a system to find the minimizer of the proposed model.
\subsection{Review}
We first review the classical SIR model.
\begin{equation*}
\left\{ \begin{aligned}
        \frac{dS(t)}{dt} &= - \beta S(t) I(t)\\
        \frac{dI(t)}{dt} &= \beta S(t) I(t) - \gamma I(t)\\
        \frac{dR(t)}{dt} &= \gamma I(t)
    \end{aligned}\right.
\end{equation*}
where $S, I, R:\ [0,T]\rightarrow [0,1]$ represent the proportion of the susceptible population, infected population, and recovered population, respectively, given time $t\in[0,T]$. Susceptible people become infected with a rate $\beta$ and infected people are recovered with a rate $\gamma$.  The SIR model can be derived based on the mean-field assumptions, thus it can be interpreted as the mean field equations for a three-state Markov chain on $S$, $I$, $R$ states.

\subsection{Spatial SIR variational problem}
We consider the spatial dimension of the $S$, $I$, $R$ functions. Let $\Omega\subset \mathbb{R}^d$ be a bounded domain. Consider the following functions
\begin{equation*}
    \rho_S, \rho_I, \rho_R:\ [0,T]\times \Omega \rightarrow [0,\infty).
\end{equation*}
Here, $\rho_S$, $\rho_I$, and $\rho_R$ represent susceptible, infected and recovered populations, respectively. We assume $\rho_i$ for each $i\in \{S,I,R\}$ moves on a spatial domain $\Omega$ with velocities $v_i$. We can describe these movements by continuity equations.
\begin{equation}\label{eq:continuity}
  \left\{ \begin{aligned}
       & \partial_t \rho_S + \nabla \cdot (\rho_S v_S) + \beta \rho_S \rho_I - \frac{\eta_S^2}{2} \Delta \rho_S= 0\\
        &\partial_t \rho_I + \nabla \cdot (\rho_I v_I) - \beta \rho_S \rho_I + \gamma \rho_I - \frac{\eta^2_I}{2} \Delta \rho_I = 0\\
        &\partial_t \rho_R + \nabla \cdot (\rho_R v_R) - \gamma \rho_I - \frac{\eta^2_R}{2} \Delta \rho_R= 0\\
        & \rho_S(0,\cdot), \rho_I(0,\cdot), \rho_R(0,\cdot) \text{ are given.}
    \end{aligned}\right.
\end{equation}
where $v_i: [0,T]\times\Omega \rightarrow \mathbb{R}^d$ ($i\in \{S,I,R\}$) are vector fields that represent the velocity fields for $\rho_i$ ($i\in \{S,I,R\}$) and nonnegative constants $\eta_i$ ($i\in \{S,I,R\}$) are coefficients for viscosity terms. We add these viscosity terms to regularize the systems of continuity equations, thus stabilize our numerical method that will be discussed in later sections. In addition, we assume zero flux conditions by the Neumann boundary conditions, that is no mass can flow in or out of $\Omega$. These systems of continuity equations satisfy the following equality:
\[
    \frac{\partial}{\partial t} \int_\Omega \rho_S(t,x)+\rho_I(t,x)+\rho_R(t,x) dx= 0,
\]
i.e., the total mass of the three populations will be conserved for all time.

Lastly, we introduce the proposed mean field control models. 
Consider the following variational problem:
\begin{equation}\label{var1}
    \begin{split}
        &\inf_{(\rho_i, v_i)_{i\in \{S,I,R\}}} \indent E(\rho_I(T,\cdot)) + \int^T_0  \int_\Omega \sum_{i\in \{S,I,R\}} \frac{\alpha_i}{2}\rho_i |v_i|^2 + \frac{c}{2}(\rho_S + \rho_I + \rho_R)^2 dx dt\\
        &\text{ subject to \eqref{eq:continuity} with fixed initial densities.}
    \end{split}
\end{equation}
Here $E$ is a convex functional and $\alpha_i$ ($i\in \{S,I,R\}$) and $c$ are nonnegative constants. The formulation is mainly divided into two parts: a terminal cost and a running cost. The functional $E$ is a terminal cost which increases if there is greater mass of infected population at the terminal time. For example, we choose $E(\rho(T,\cdot)) = \frac{1}{2}\int_\Omega \rho^2(T,x) dx$ for the experiments (Section~\ref{section4}). The rest of the terms besides the functional $E$ are running costs. Kinetic energy terms $\frac{\alpha_i}{2} \rho_i |v_i|^2$~$(i\in\{S,I,R\})$ represent the cost of moving the density $\rho_i$ with velocities $v_i$ over time $0\leq t \leq T$. A high value of $\alpha_i$ means it is expensive to move $\rho_i$ for corresponding $i\in\{S,I,R\}$. In the numerical experiments (Section~\ref{section4}), we assume $\alpha_S=\alpha_R=1$ and $\alpha_I=10$ to simulate the real life scenario where infected group is harder to move than other groups. The last term in the running cost, $\frac{c}{2}(\rho_S+\rho_I+\rho_R)^2$, penalizes congestion of the total population. A high value of $c$ means more penalization on the congestion. The minimizers of the variational problem will provide the optimal movements for each population while minimizing the terminal cost functional with respect to the infected population $\rho_I$. 

We note that the function $(\rho_i,v_i) \mapsto \rho_i |v_i|^2$ is not convex. By introducing new variables $m_i := \rho_i v_i$, we convert the cost function to be convex.
\begin{subequations}\label{eq:primal}
    \begin{equation}
        \min_{\rho_i, v_i} \indent P(\rho_i,m_i)_{i\in \{S,I,R\}} 
    \end{equation}    
        subject to
               \begin{equation}
        \left\{    \begin{aligned}
&\partial_t \rho_S + \nabla \cdot m_S + \beta \rho_S \rho_I - \frac{\eta_S^2}{2} \Delta\rho_S = 0\\
 &\partial_t \rho_I + \nabla \cdot m_I - \beta \rho_S \rho_I + \gamma \rho_I - \frac{\eta^2_I}{2} \Delta\rho_I = 0\\
& \partial_t \rho_R + \nabla \cdot m_R - \gamma \rho_I - \frac{\eta^2_R}{2} \Delta \rho_R= 0\\
&\rho_S(0,\cdot), \rho_I(0,\cdot), \rho_R(0,\cdot) \text{ are given}
    \end{aligned}\right.
\end{equation}
\end{subequations}
where
\begin{align*}
    P(\rho_i,m_i)_{i\in \{S,I,R\}} =& E(\rho_I(T,\cdot)) + \int^T_0  \int_\Omega \sum_{i\in \{S,I,R\}} \frac{\alpha_i |m_i|^2}{2\rho_i}  + \frac{c}{2}(\rho_S + \rho_I + \rho_R)^2 dx dt.
\end{align*}
From an optimization viewpoint, we note that the minimization problem is not a convex problem since the coupling terms, $\beta \rho_S \rho_I$, in constraints make the feasible set nonconvex. 
We replace the nonconvex coupling term $\beta\rho_S\rho_I$ with convolution. Note that Kendall \cite{kendall1965mathematical} introduced this kernel for modeling pandemic dynamics and took the nonlocal exposure to infectious agents into consideration. This term also helps regularize the minimization problem.
\begin{subequations}\label{eq:primal-kernel}
\begin{equation}
\min_{(\rho_i, v_i)_{i\in \{S,I,R\}}} \indent P(\rho_i, m_i)_{i\in \{S,I,R\}}
\end{equation}
subject to
\begin{equation}\label{eq:primal-kernel-constraint}
   \left\{ \begin{aligned}
        &  \partial_t \rho_S(t,x) + \nabla \cdot m_S(t,x) + \beta \rho_S(t,x) \int_\Omega K(x,y) \rho_I(t,y)dy - \frac{\eta_S^2}{2} \Delta\rho_S(t,x) = 0\\
        & \partial_t \rho_I(t,x) + \nabla \cdot m_I(t,x) - \beta \rho_I(t,x) \int_\Omega K(x,y)\rho_S(t,y) dy + \gamma \rho_I(t,x) - \frac{\eta^2_I}{2} \Delta\rho_I(t,x) = 0\\
        &  \partial_t \rho_R(t,x) + \nabla \cdot m_R(t,x) - \gamma \rho_I(t,x) - \frac{\eta^2_R}{2} \Delta \rho_R(t,x) = 0\\
        &  \rho_S(0,\cdot), \rho_I(0,\cdot), \rho_R(0,\cdot) \text{ given.}
    \end{aligned}\right.
\end{equation}
{Here, $K(x,y)$ is a symmetric positive definite kernel.}
\end{subequations}
In this paper, we focus on a Gaussian kernel
\[
    K(x,y) = \frac{1}{\sqrt{(2\pi)^d}} \prod^d_{k=1} \frac{1}{\sigma_k} \exp{\left(-\frac{|x_k-y_k|^2}{2\sigma_k^2}\right)}.
\]
The variance $\sigma_k$ of Gaussian kernel can be viewed as a parameter for modeling the spatial spreading effect of virus. Let's consider the convolution term in the first continuity equation, $\rho_S(t,x) \int_\Omega K(x,y)\rho_I(t,y)dy$. Larger values of variance $\sigma_k$'s in $K$ mean a susceptible agent located at position $x$ can be affected by infectious agents farther away from $x$. Note that by letting $\sigma_k\rightarrow 0$, we get
\[
    \rho_S(t,x) \int_\Omega K(x,y)\rho_I(t,y)dy \rightarrow \rho_S(t,x) \rho_I(t,x).
\]
Thus, when $\sigma_k$ becomes close to $0$, the susceptible agent is only affected by infectious agents nearby. If we let $\sigma_k \rightarrow \infty$, then
\[
    \rho_S(t,x) \int_\Omega K(x,y)\rho_I(t,y)dy \rightarrow \rho_S(t,x) \int_\Omega \rho_I(t,y) dy,
\]
which means the susceptible group is affected by the total number of infected population.

\begin{remark}
    The formulation is not limited to the SIR model we chose in this paper. It can be used to solve any types of spatial epidemiological models. For example, if we use SEIR model where E stands for exposed group, we just add one additional variable $\rho_E$ and add one more continuity equation.
\end{remark}

\subsection{Properties}
We next derive the mean field control system, i.e. the minimizer system associated with spatial SIR variational problem \eqref{eq:primal-kernel}. We introduce three dual variables $\phi_i$ $(i\in\{S,I,R\})$ to convert the minimization problem~\eqref{eq:primal-kernel} into a saddle problem.
\begin{equation*}
    \begin{aligned}
        &\inf_{(\rho_i,v_i)_{i\in\{S,I,R\}}} \left\{P(\rho_i,m_i)_{i\in\{S,I,R\}}: \text{subject to}~\eqref{eq:primal-kernel-constraint} \right\}
    \end{aligned}
\end{equation*}
\begin{equation*}
    \begin{aligned}
        =& \inf_{(\rho_i,v_i)_{i\in\{S,I,R\}}} \sup_{(\phi_i)_{i\in\{S,I,R\}}} P(\rho_i,m_i)_{i\in\{S,I,R\}}\\
        &\quad\quad - \int^T_0 \int_\Omega \phi_S \left( \partial_t \rho_S + \nabla \cdot m_S + \beta \rho_S K * \rho_I - \frac{\eta_S^2}{2} \Delta \rho_S \right) dx dt\\
        &\quad\quad - \int^T_0 \int_\Omega \phi_I \left( \partial_t \rho_I + \nabla \cdot m_I - \beta \rho_S K * \rho_I + \gamma \rho_I - \frac{\eta_I^2}{2} \Delta \rho_I \right) dx dt\\
        &\quad\quad - \int^T_0 \int_\Omega \phi_R \left( \partial_t \rho_R + \nabla \cdot m_R - \gamma \rho_I - \frac{\eta_R^2}{2} \Delta \rho_R \right) dx dt.
    \end{aligned}
\end{equation*}
Simplifying the above function, we define the Lagrangian functional
\begin{equation}\label{eq:lagrangian}
\begin{split}
    &\mathcal{L}((\rho_i,m_i,\phi_i)_{i\in \{S,I,R\}})\\
    = & P(\rho_i, m_i)_{i\in \{S,I,R\}}
    -  \int^T_0 \int_\Omega
     \sum_{i\in \{S,I,R\}} \phi_i \left( \partial_t \rho_i + \nabla \cdot m_i - \frac{\eta_i^2}{2} \Delta \rho_i \right) dx dt\\
        & + \int^T_0\int_\Omega \beta \rho_S (\phi_I - \phi_S) K * \rho_I + \gamma\rho_I (\phi_R - \phi_I) dx dt.
        \end{split}
\end{equation}
Thus, we have the following saddle problem:
\begin{equation}\label{eq:saddle}
    \inf_{(\rho_i,m_i)_{i\in \{S,I,R\}}} \sup_{(\phi_i)_{i\in \{S,I,R\}}} \mathcal{L}((\rho_i,m_i,\phi_i)_{i\in \{S,I,R\}}).
\end{equation}
The existence of the saddle point of this minimax problem is based on the assumption that the dual gap is zero. In other words, given a primal solution with respect to optimal primal variables $(\rho_i^*,m_i^*)_{i\in\{S,I,R\}}$ and a dual solution with respect to optimal dual variables $(\phi_i^*)_{i\in\{S,I,R\}}$, the difference between these two solutions is zero. However, the dual gap may not be zero for this problem because the nonconvex functional $(\rho_S,\rho_I) \mapsto \rho_S K * \rho_I$ makes feasible set of the problem nonconvex. Throughout the paper, we assume the dual gap is zero to get properties of saddle points.

    The following propositions are the properties of the saddle point problem derived from optimality conditions (Karush–Kuhn–Tucker (KKT) conditions).
\begin{proposition}[Mean-field control SIR system]\label{proposition:KKT-conditions}
By KKT conditions, the saddle point $(\rho_i^*, m_i^*, \phi_i^*)$ of \eqref{eq:saddle} satisfies the following equations.
    \begin{equation}\label{MFSIR}
        \left\{\begin{aligned}
            &\partial_t \phi_S^* - \frac{\alpha_S}{2} |\nabla \phi_S^*|^2 + \frac{\eta^2_S}{2}\Delta \phi_S^* + c(\rho_S^*+\rho_I^*+\rho_R^*) + \beta (\phi_I^* - \phi_S^*) K * \rho_I^*  = 0\\
            &\partial_t \phi_I^* - \frac{\alpha_I}{2} |\nabla \phi_I^*|^2 + \frac{\eta^2_I}{2}\Delta \phi_I^* + c(\rho_S^*+\rho_I^*+\rho_R^*)\\
            &\hspace{4.5cm} + \beta K * \left( \rho_S (\phi_I - \phi_S) \right) + \gamma (\phi_R^* - \phi_I^*) = 0\\
            &\partial_t \phi_R^* - \frac{\alpha_R}{2} |\nabla \phi_R^*|^2 + \frac{\eta^2_R}{2}\Delta \phi_R^* + c(\rho_S^*+\rho_I^*+\rho_R^*)= 0\\
            & \partial_t \rho_S^* -\frac{1}{\alpha_S}\nabla \cdot (\rho_S^*\nabla\phi_S^*) + \beta \rho_S^* K * \rho_I^* - \frac{\eta_S^2}{2} \Delta\rho_S^* = 0\\
            &\partial_t \rho_I^* -\frac{1}{\alpha_I}\nabla \cdot(\rho_I^* \nabla\phi_I^*) - \beta \rho_S^* K *\rho_I^* + \gamma \rho_I^* - \frac{\eta^2_I}{2} \Delta\rho_I^* = 0\\
            & \partial_t \rho_R^* -\frac{1}{\alpha_R} \nabla \cdot (\rho_R^* \nabla\phi_R^*) - \gamma \rho_I^* - \frac{\eta^2_R}{2} \Delta \rho_R^*= 0\\
            &\phi_I^*(T,\cdot) = \delta E (\rho_I^*(T,\cdot)).
        \end{aligned}\right.
    \end{equation}
\end{proposition}
\begin{proof}
By integration by parts, we reformulate the Lagrangian function \eqref{eq:saddle} as follows. 
    \begin{equation*}
        \begin{aligned}
        &L((\rho_i,m_i,\phi_i)_{i\in \{S,I,R\}}) \\ =&E(\rho_I(T,\cdot)) + \int^T_0\int_\Omega \frac{c}{2}(\rho_S + \rho_I + \rho_R)^2 + \beta \rho_S (\phi_I - \phi_S) K * \rho_I + \gamma\rho_I (\phi_R - \phi_I) dx dt\\
            & + \sum_{i\in \{S,I,R\}} \int^T_0 \int_\Omega \frac{\alpha_i |m_I|^2}{2 \rho_i} + \rho_i \partial_t \phi_i + m_i \cdot \nabla \phi_i + \frac{\eta_i^2}{2} \rho_i \Delta \phi_i dx dt\\
            & + \sum_{i\in \{S,I,R\}} \int_\Omega \rho_i(0,x) \phi_i(0,x) - \rho_i(T,x) \phi_i(T,x) dx
        \end{aligned}
    \end{equation*}
    If $(\rho_i^*, m_i^*, \phi_i^*)$ are saddle points of the Lagrangian, the differential of Lagrangian with respect to $\rho_i$, $m_i$, $\phi_i$ ($i\in \{S,I,R\}$) and $\rho_I(T,\cdot)$ equal to zero. We have
    \begin{equation*}
    \begin{aligned}
    \frac{\delta L}{\delta\rho_i}(\rho_i^*, m_i^*, \phi_i^*) =
        \frac{\delta L}{\delta m_i}(\rho_i^*, m_i^*, \phi_i^*)=
        \frac{\delta L}{\delta\phi_i}(\rho_i^*, m_i^*, \phi_i^*)=
        \frac{\delta L}{\delta\rho_I(T,\cdot)}(\rho_i^*, m_i^*, \phi_i^*)=0.
    \end{aligned}
    \end{equation*}
    From $\frac{\delta L}{\delta\rho_I(T,\cdot)}(\rho_i^*, m_i^*, \phi_i^*)=0$ and $\frac{\delta L}{\delta m_i}(\rho_i^*, m_i^*, \phi_i^*)=0$, we have $\phi_I^*(T,\cdot) = \delta E (\rho_I^*(T,\cdot))$ and $m_i^* = - \frac{1}{\alpha_i} \rho_i^*\nabla \phi_i^*$ ($i\in\{S,I,R\}$), respectively. Plugging in these equations into the rest of the equations, we derive the result.
\end{proof}
We note that dynamical system \eqref{MFSIR} models the optimal vector field strategies for S, I, R populations. It combines both strategies from mean field controls and SIR models. For this reason, we call \eqref{MFSIR} {\em Mean-field control SIR system}.
\section{Algorithm}\label{section3}
In this section, we implement optimization methods to solve the proposed SIR variational problems. Specifically, we use G-Prox Primal Dual Hybrid Gradient (G-Prox PDHG) method \cite{JacobsLegerLiOsher2018_solvinga}. This is a variation of Chambolle-Pock primal-dual algorithm \cite{champock11,champock16}. G-Prox PDHG proposes a way of choosing proper norms for the optimization  based on the given minimization problem whereas Chambolle-Pock primal-dual algorithm just uses $L^2$ norms. Choosing appropriate norms results in faster and more robust convergence of the algorithm.

\subsection{Review of primal-dual algorithms}

The PDHG method solves the minimization problem
\[
    \min_x f(Ax) + g(x)
\]
by converting it into a saddle point problem
\[
    \min_x \sup_y \left\{ L(x,y) := \langle Ax, y \rangle + g(x) - f^*(y) \right\}.
\]
Here, $f$ and $g$ are convex functions with respect to a variable $x$, $A$ is a continuous linear operator, and
\[
    f^*(y) = \sup_{x} x\cdot y - f(x)
\] 
is a Legendre transform of $f$. For each iteration, the algorithm finds the minimizer $x_*$ by gradient descent method and the maximizer $y_*$ by gradient ascent method. Thus, the minimizer and maximizer are calculated by iterating
\begin{equation*}
    \begin{cases}
        x^{k+1} &= \argmin_x L(x,y^k) + \frac{1}{2\tau} \|x - x^k\|^2\\
        y^{k+\frac{1}{2}} &= \argmax_y L(x^{k+1},y) + \frac{1}{2\sigma} \|y - y^k\|^2\\
        y^{k+1} &= 2 y^{k+\frac{1}{2}} - y^k
    \end{cases}
\end{equation*}
where $\tau$ and $\sigma$ are step sizes for the algorithm.

G-Prox PDHG is a modified version of PDHG that solves the minimization problem by choosing the most appropriate norms for updating $x$ and $y$. Choosing the appropriate norms allows us to choose larger step sizes. Hence, we get a faster convergence rate. In details, 
\begin{equation*}
    \begin{cases}
        x^{k+1} &= \argmin_x L(x,y^k) + \frac{1}{2\tau} \|x - x^k\|^2_{\mathcal{H}}\\
        y^{k+\frac 1 2} &= \argmax_y L(x^{k+1},y) + \frac{1}{2\sigma} \|y - y^k\|^2_{\mathcal{G}}\\
        y^{k+1} &= 2 y^{k+\frac{1}{2}} - y^k
    \end{cases}
\end{equation*}
where $\mathcal{H}$ and $\mathcal{G}$ are some Hilbert spaces with the inner product
\[
    (u_1,u_2)_{\mathcal{G}} = (A u_1, A u_2)_{\mathcal{H}}.
\] 
In particular, we use G-Prox PDHG to solve the minimization problem \eqref{eq:primal-kernel} by setting $\mathcal{H} = L^2$ and $\mathcal{G} = H^2$. Furthermore,
\begin{equation*}
x = (\rho_S, \rho_I, \rho_R, m_S, m_I, m_R), \quad
    g(x) = P(\rho_i,m_i)_{i\in \{S,I,R\}},\quad f(Ax) = 
    \begin{cases}
     0 & \text{if } Ax = (0,0,\gamma\rho_I)\\
     \infty & \text{otherwise.}
    \end{cases}
\end{equation*}
\begin{equation*}
    \begin{aligned}
        Ax = (& \partial_t \rho_S + \nabla \cdot m_S - \frac{\eta^2}{2} \Delta \rho_S + \beta\rho_S K * \rho_I,\\
              & \partial_t \rho_I + \nabla \cdot m_I - \frac{\eta^2}{2} \Delta \rho_I - \beta\rho_I K * \rho_S + \gamma \rho_I,\\
              &\partial_t \rho_R + \nabla \cdot m_R - \frac{\eta^2}{2} \Delta \rho_R ).
    \end{aligned}
\end{equation*}
Thus, we have the following inner products
\begin{align*}
    (u_1,u_2)_{L^2} = \int^T_0\int_\Omega u_1(t,x) u_2(t,x) dx dt, \quad (u_1,u_2)_{H^2} = \int^T_0 \int_\Omega A u_1(t,x) A u_2(t,x) dx dt.
\end{align*}
Note that the operator $A$ is not linear. In the implementation, we approximate the operator with the following linear operator
\begin{equation*}
    \begin{aligned}
        Ax \approx (& \partial_t \rho_S + \nabla \cdot m_S - \frac{\eta^2}{2} \Delta \rho_S + \beta\rho_S,\\
              & \partial_t \rho_I + \nabla \cdot m_I - \frac{\eta^2}{2} \Delta \rho_I + (\gamma + \beta) \rho_I,\\
              &\partial_t \rho_R + \nabla \cdot m_R - \frac{\eta^2}{2} \Delta \rho_R ).
    \end{aligned}
\end{equation*}

\subsection{G-Prox PDHG on SIR variational problem}

In this section, we implement G-Prox PDHG to solve the saddle problem \eqref{eq:saddle}. For $i\in \{S,I,R\}$,
\begin{equation*}
\begin{split}
    \rho_i^{(k+1)} &= \argmin_\rho \mathcal{L}(\rho, m_i^{(k)},\phi_i^{(k)}) + \frac{1}{2\tau_i} \|\rho - \rho_i^{(k)}\|^2_{L^2}\\
    m_i^{(k+1)} &= \argmin_m \mathcal{L}(\rho_i^{(k+1)}, m,\phi_i^{(k)}) + \frac{1}{2\tau_i} \|m - m_i^{(k)}\|^2_{L^2}\\
    \phi_i^{(k+ \frac 1 2)} &= \argmax_\phi \mathcal{L}(\rho_i^{(k+1)}, m_i^{(k+1)},\phi) - \frac{1}{2\sigma_i} \|\phi - \phi_i^{(k)}\|^2_{H^2}\\
    \phi_i^{(k+1)} &= 2 \phi_i^{(k+\frac 1 2)} - \phi_i^{(k)}
 \end{split}
\end{equation*}
where $\tau_i$, $\sigma_i$ ($i\in \{S,I,R\}$) are step sizes for the algorithm and by G-Prox PDHG, $L^2$ norm and $H^2$ norm are defined as
\begin{align*}
        \|u\|^2_{L^2} = \int^T_0 \int_\Omega u^2 dx dt, \quad \|u\|^2_{H^2} = \int^T_0 \int_\Omega (\partial_t u)^2 + |\nabla u|^2 + \frac{\eta^4}{4} (\Delta u)^2 dx dt
\end{align*}
for any $u : [0,T] \times \Omega \rightarrow [0,\infty)$. \\

By formulating these optimality conditions, we can find explicit formulas for each variable.

\begin{equation*}
\begin{split}
    \rho_S^{(k+1)} &= root_+\Biggl(\frac{\tau_S}{1+ c\tau_S} \biggl(\partial_t\phi_S^{(k)} + \frac{\eta_S^2}{2} \Delta\phi_S^{(k)} - \frac{1}{\tau_S} \rho_S^{(k)} + \beta\left( K*(\phi_I^{(k)} \rho_I^{(k)}) - \phi_S^{(k)} K * \rho_I^{(k)} \right)\\
    & \hspace{1.5cm}  + c(\rho_I + \rho_R)\biggl), 0, - \frac{\tau_S \alpha_S (m_S^{(k)})^2}{2(1+c\tau_S)}
    \Biggl)\\
    \rho_I^{(k+1)} &= root_+\Biggl(\frac{\tau_I}{1+ c\tau_I} \biggl(\partial_t\phi_I^{(k)} + \frac{\eta_I^2}{2} \Delta\phi_I^{(k)} - \frac{1}{\tau_I} \rho_I^{(k)} + \beta\left( \phi_I^{(k)} K * \rho_S^{(k)} - K * (\phi_S^{(k)} \rho_S^{(k)}) \right)\\
    & \hspace{1.5cm}  + \gamma (\phi_R - \phi_I) + c(\rho_S + \rho_R)\biggl), 0, - \frac{\tau_I \alpha_I (m_I^{(k)})^2}{2(1+c\tau_I)}
    \Biggl)\\
    \rho_R^{(k+1)} &= root_+\Biggl(\frac{\tau_R}{1+ c\tau_R} \biggl(\partial_t\phi_R^{(k)} + \frac{\eta_R^2}{2} \Delta\phi_R^{(k)} - \frac{1}{\tau_R} \rho_R^{(k)} + c(\rho_S + \rho_I)\biggl), 0, - \frac{\tau_R \alpha_R (m_R^{(k)})^2}{2(1+c\tau_R)}
    \Biggl)
\end{split}
\end{equation*}

\begin{flalign*}
   m_i^{(k+1)} &= \frac{\rho_i^{(k+1)}}{\tau\alpha_i + \rho_i^{(k+1)}} \left( m_i^{(k)} - \tau \nabla \phi_i^{(k)} \right),\indent (i\in \{S,I,R\})
\end{flalign*}
\begin{flalign*}
   &\phi_S^{(k+1)} = \phi_S^{(k)} + \sigma_S (A_S^T A_S)^{-1} \left( -\partial_t \rho^{(k+1)}_S - \nabla \cdot m^{(k+1)}_S - \beta \rho^{(k+1)}_S K * \rho^{(k+1)}_I  + \frac{\eta_S^2}{2} \Delta \rho^{(k+1)}_S \right)&
\end{flalign*}
\begin{flalign*}
    \phi_I^{(k+\frac 1 2)} &= \phi_I^{(k)} + \sigma_I (A_I^T A_I)^{-1} \biggl( -\partial_t \rho^{(k+1)}_I - \nabla \cdot m^{(k+1)}_I + \beta  \rho^{(k+1)}_I K * \rho^{(k+1)}_S\\
    &\hspace{7.2cm} - \gamma \rho^{(k+1)}_I  + \frac{\eta_I^2}{2} \Delta \rho^{(k+1)}_I \biggl)&
\end{flalign*}

\begin{flalign*}
    \phi_R^{(k+\frac 1 2)} &= \phi_R^{(k)} + \sigma_R (A_R^T A_R)^{-1} \left( -\partial_t \rho^{(k+1)}_R - \nabla \cdot m^{(k+1)}_R +  \gamma \rho^{(k+1)}_I + \frac{\eta_R^2}{2} \Delta \rho^{(k+1)}_R\right)&
\end{flalign*}
where $root_+(a,b,c)$ is a positive root of a cubic polynomial $x^3 + a x^2 + b x +c = 0$ and 
\begin{equation*}
    \begin{split}
        A_S^{T}A_S &= -\partial_{tt} + \frac{\eta_S^4}{4} \Delta^2 - (1 + 2 \beta \eta_S) \Delta + \beta^2\\
        A_I^{T}A_I &= -\partial_{tt} + \frac{\eta_I^4}{4} \Delta^2 - (1 + 2 (\gamma+\beta) \eta_S) \Delta + (\gamma + \beta)^2\\
        A_R^{T}A_R &= -\partial_{tt} + \frac{\eta_R^4}{4} \Delta^2 - \Delta.
    \end{split}
\end{equation*}
We use FFTW library to compute $(A_i^T A_i)^{-1}$ ($i\in \{S,I,R\}$) and convolution terms by Fast Fourier Transform (FFT), which is $O(n\log n)$ operations per iteration where $n$ is the number of points. Thus, the algorithm takes just $O(n\log n)$ operations per iteration.

In all, we summarize the algorithm as follows. 
\begin{tabbing}
aaaaa\= aaa \=aaa\=aaa\=aaa\=aaa=aaa\kill  
   \rule{\linewidth}{0.8pt}\\
   \noindent{\large\bf Algorithm: G proximal PDHG for mean-field control SIR system}\\
  \1 \textbf{Input}: $\rho_i(0,\cdot)$ ($i\in \{S,I,R\}$)\\
  \1 \textbf{Output}: $\rho_i, m_i, \phi_i$ ($i\in \{S,I,R\}$) for $x\in\Omega$, $t\in[0,T]$\\
   \rule{\linewidth}{0.5pt}\\
  \1 \textbf{While} relative error $>$ tolerance  \\
\2 $\rho_i^{(k+1)} = \argmin_\rho \mathcal{L}(\rho, m_i^{(k)},\phi_i^{(k)}) + \frac{1}{2\tau_i} \|\rho - \rho_i^{(k)}\|^2_{L^2}$\\
\2     $m_i^{(k+1)} = \argmin_m \mathcal{L}(\rho^{(k+1)}, m,\phi_i^{(k)}) + \frac{1}{2\tau_i} \|m - m_i^{(k)}\|^2_{L^2}$\\
\2     $\phi_i^{(k+ \frac 1 2)} = \argmax_\phi \mathcal{L}(\rho^{(k+1)}, m_i^{(k+1)},\phi) - \frac{1}{2\sigma_i} \|\phi - \phi_i^{(k)}\|^2_{H^2}$\\
\2     $\phi_i^{(k+1)} = 2 \phi_i^{(k+\frac 1 2)} - \phi_i^{(k)}$\\
  \1 \End\\
   \rule{\linewidth}{0.8pt}
\end{tabbing}
Here, the relative error is defined as
\[
\text{relative error} = \frac{|P(\rho_i^{(k+1)},m_i^{(k+1)})- P(\rho_i^{(k)},m_i^{(k)})|}{|P(\rho_i^{(k)},m_i^{(k)})|}.
\]

\section{Experiments}\label{section4}
In this section, we present several sets of numerical experiments using the algorithm with various parameters. We wrote C++ codes to run the numerical experiments. Let $\Omega = [0,1]^2$ be a unit cube in $\mathbb{R}^2$ and $T=1$. The domain $\Omega$ is discretized with the regular rectangular mesh
\[\Delta x = \frac{1}{N_x},\quad \Delta y = \frac{1}{N_y},\quad \Delta t = \frac{1}{N_t-1}\]
\begin{equation*}
    \begin{aligned}
        x_{kl} &= \left( (k+0.5) \Delta x, (l+0.5)\Delta y \right), && k = 0,\cdots,N_x-1,\quad l = 0,\cdots,N_y-1\\
        t_n &= n \Delta t , && n = 0,\cdots,N_t-1
    \end{aligned}
\end{equation*}
where $N_x$, $N_y$ are the number of data points in space and $N_t$ is the number of data points in time.
For all the experiments, we use the same set of parameters,
\[
    N_x = 128, \quad N_y = 128,\quad N_t = 32
\]
\[
    \sigma=0.02,\quad c = 0.01,\quad \eta_i = 0.01 \quad(i\in \{S,I,R\})
\]
\[
    \alpha_S = 1, \quad \alpha_I = 10, \quad \alpha_R = 1
\]
and choose the same terminal cost functional
\[
    E(\rho_I(1,\cdot)) = \frac 1 2 \int_\Omega \rho_I^2(1,x) dx.
\]
By setting higher value for $\alpha_I$, we penalize the movement of infected population more than other populations. Considering the immobility of infected individuals, this is a reasonable choice in terms of real-world applications.

To minimize the terminal cost functional $E(\rho_I(T,\cdot))$, a solution needs to reduce the number of infected population. There are mainly two ways of reducing the number of infected. One way is to recover infected to recovered population. However, it may not be feasible if a rate of recovery $\gamma$ is small. Another way to reduce the number of infected is by separating susceptible population from infected population. The number of infected doesn't increase if there are no susceptible people near infected. However, the total cost increases when densities move due to the kinetic energy term $\rho_i |v_i|^2$ ($i\in \{S,I,R\}$) in the running cost. A solution needs to find the optimal balance between the terminal cost and the running cost. Experiment~1 shows the effectiveness of controlling populations' movements. We compute two solutions of the model: with and without control of movements. The comparison between these solutions shows that the number of infected people at the terminal time can be reduced effectively with control. Experiment~2 shows that the algorithm finds the proper solutions based on different recovery rates given nonsymmetric initial densities. In Experiment~3, we consider a more complicating terminal energy functional $E(\rho_I(T,\cdot))$, and compute the solutions based on different infection rates.

\subsection{Experiment 1}

In this experiment, we compare the solutions of SIR model with and without control. We set initial densities for susceptible, infected and recovered populations as
\begin{equation*}
    \begin{aligned}
        \rho_S(0,x=(x_1,x_2)) &= 0.6 \exp{\Bigl(-10\bigl((x_1-0.5)^2 + (x_2-0.5)^2 \bigr)\Bigr)}\\
        \rho_I(0,x=(x_1,x_2)) &= 0.6 \exp{\Bigl(-35\bigl((x_1-0.6)^2 + (x_2-0.6)^2 \bigr)\Bigr)}\\
        \rho_R(0,x=(x_1,x_2)) &= 0
    \end{aligned}
\end{equation*}
Susceptible population and infected population are Gaussian distributions centered at $(0.5,0.5)$ and $(0.6,0.6)$, respectively. We set $\beta = 0.7$ and $\gamma = 0.1$.

We show two different numerical results: one with control and one without control. The formulation without control has the following system of equations,
\begin{equation*}
\begin{aligned}
        \frac{\partial\rho_S(t,x)}{dt} &= - \beta \rho_S(t,x) \rho_I(t,x)\\
        \frac{\partial\rho_I(t,x)}{dt} &= \beta \rho_S(t,x) \rho_I(t,x) - \gamma \rho_I(t,x)\\
        \frac{\partial\rho_R(t,x)}{dt} &= \gamma \rho_I(t,x).
    \end{aligned}
\end{equation*}
By removing the velocity terms, we assume no movements of population. We solve these equations by using Euler's method. Thus, the solution can be computed by iterating $n=0,\cdots,N_t-2$,
\begin{align*}
    \rho_S(t_{n+1},x_{kl}) &= \rho_S(t_n,x_{kl}) - \Delta t \beta  \rho_S(t_n,x_{kl}) \rho_I(t_n,x_{kl})\\
    \rho_I(t_{n+1},x_{kl}) &= \rho_I(t_n,x_{kl}) + \Delta t \left(\beta  \rho_S(t_n,x_{kl}) \rho_I(t_n,x_{kl}) - \gamma \rho_I(t_n,x_{kl}) \right)\\
    \rho_R(t_{n+1},x_{kl}) &= \rho_R(t_n,x_{kl}) + \Delta t \gamma  \rho_I(t_n,x_{kl}),
\end{align*}
for $k = 0,\cdots,N_x-1, l=0,\cdots,N_y-1$. The results can be seen in Figure~\ref{fig:exp1-terminal-time-sir} and Figure~\ref{fig:sir-comparison-graphs}. Figure~\ref{fig:exp1-terminal-time-sir} shows snapshots of the initial and terminal densities. The first row shows the initial densities of susceptible, infected and recovered (from left to right) based on the equations above. The second row and the third row show the terminal densities without control and with control, respectively. Figure~\ref{fig:sir-comparison-graphs} shows a quantitative comparison between these two solutions. The graphs indicate the total sum of each group over time. More specifically, they show $\int_\Omega \rho_i(t,x) dx$ for $i\in\{S,I,R\}$ from $t=0$ to $t=1$. 

In Figure~\ref{fig:exp1-terminal-time-sir}, when we compare the susceptible groups from second and third rows, the susceptible group with control moves more than the susceptible group without control. If there is no control (the second row in Figure~\ref{fig:exp1-terminal-time-sir}), the groups don't move and the susceptible group is exposed to infected group which leads to a high chance of susceptible being infected over time. If population is in control (the third row in Figure~\ref{fig:exp1-terminal-time-sir}), we see a clear separation between susceptible and infected at the terminal time. This separation decreases the exposure of susceptible to infected effectively and, as a result, we see less number of infected and more number of susceptible at the terminal time from the solution with control.

\begin{figure}[h]
    \begin{minipage}[b]{0.33\linewidth}
      \centering
      \includegraphics[width=.7\linewidth]{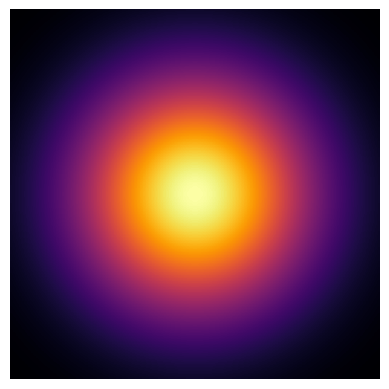}
    \end{minipage}\hfil
    \begin{minipage}[b]{0.33\linewidth}
      \centering
      \includegraphics[width=.7\linewidth]{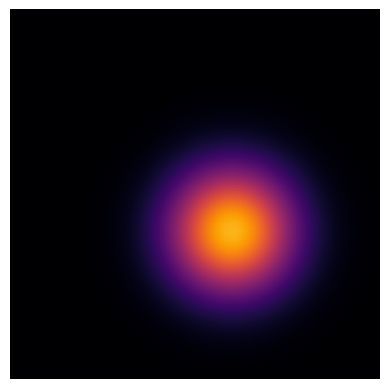}
    \end{minipage}\hfil
    \begin{minipage}[b]{0.33\linewidth}
      \centering
      \includegraphics[width=.7\linewidth]{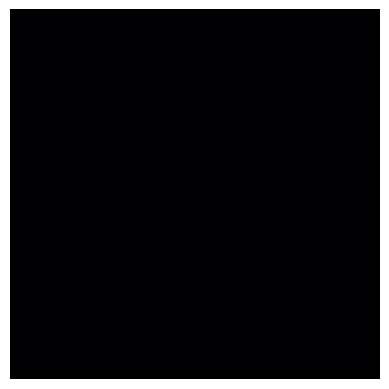}
    \end{minipage}\hfil
    \begin{minipage}[b]{0.33\linewidth}
      \centering
      \includegraphics[width=.7\linewidth]{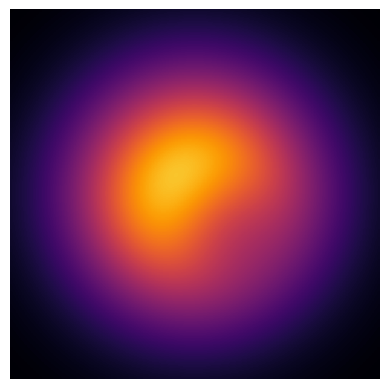}
    \end{minipage}\hfil
    \begin{minipage}[b]{0.33\linewidth}
      \centering
      \includegraphics[width=.7\linewidth]{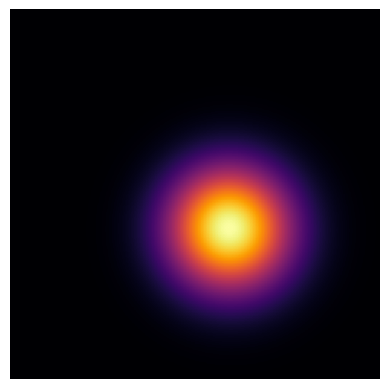}
    \end{minipage}\hfil
    \begin{minipage}[b]{0.33\linewidth}
      \centering
      \includegraphics[width=.7\linewidth]{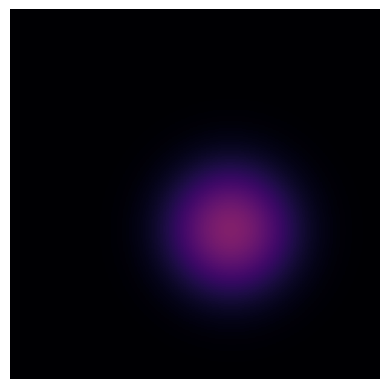}
    \end{minipage}\hfil
    \begin{minipage}[b]{0.33\linewidth}
      \centering
      \includegraphics[width=.7\linewidth]{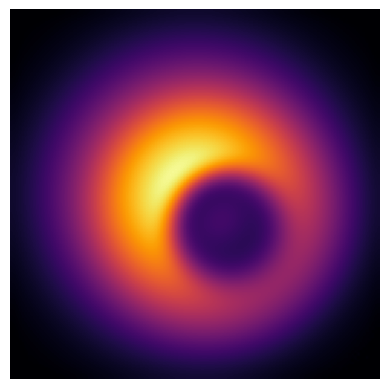}
    \end{minipage}\hfil
    \begin{minipage}[b]{0.33\linewidth}
      \centering
      \includegraphics[width=.7\linewidth]{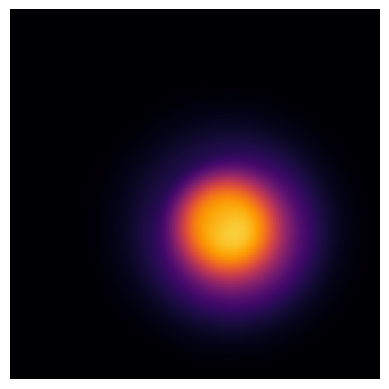}
    \end{minipage}\hfil
    \begin{minipage}[b]{0.33\linewidth}
      \centering
      \includegraphics[width=.7\linewidth]{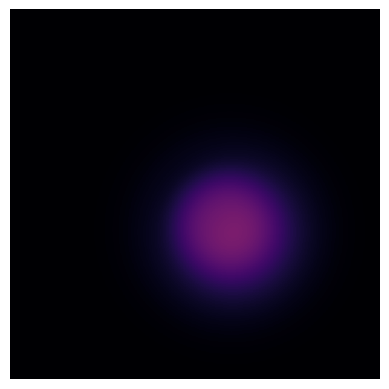}
    \end{minipage}\hfil
    \caption{Snapshots of susceptible (column 1), infected (column 2) and recovered populations (column 3). The first row shows the initial densities, the second row shows the solution without control at the terminal time and the third row shows the solution with control at the terminal time.}
    \label{fig:exp1-terminal-time-sir}
\end{figure}

\begin{figure}
     \centering
     \begin{subfigure}[b]{0.32\textwidth}
         \centering
         \includegraphics[width=\textwidth]{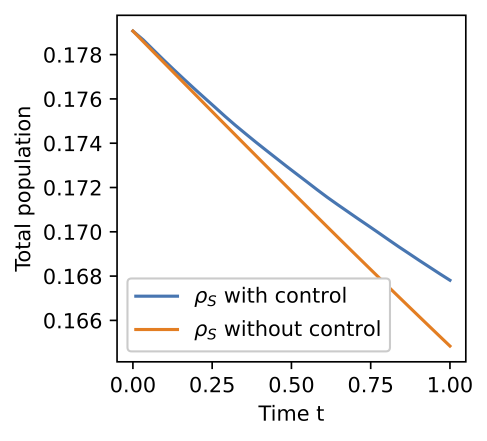}
         \caption{$\rho_S$}
         \label{fig:rho_S_comparison}
     \end{subfigure}
     \hfill
     \begin{subfigure}[b]{0.32\textwidth}
         \centering
         \includegraphics[width=\textwidth]{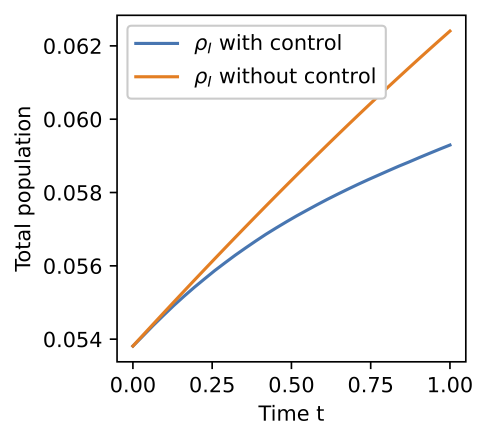}
         \caption{$\rho_I$}
         \label{fig:rho_I_comparison}
     \end{subfigure}
     \hfill
     \begin{subfigure}[b]{0.32\textwidth}
         \centering
         \includegraphics[width=\textwidth]{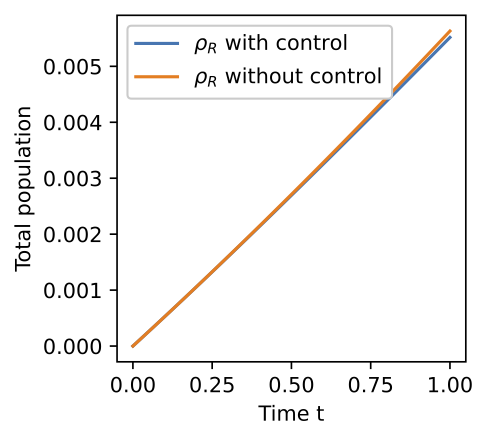}
         \caption{$\rho_R$}
         \label{fig:rho_R_comparison}
     \end{subfigure}
        \caption{The comparison between solutions with and without control. The graphs show the total population of each group $\int_\Omega \rho_i(t,x)dx$ for $0\leq t \leq 1$ and $i\in\{S,I,R\}$.}
        \label{fig:sir-comparison-graphs}
\end{figure}

\subsection{Experiment 2}

In this experiment, we consider nonsymmetric initial densities.
\begin{equation*}
    \begin{aligned}
        \rho_S(0,x) &= 0.45
        \Bigl(
        \exp\bigl(-15((x-0.3)^2+(y-0.3)^2)\bigr)\\
        &\quad + \exp\bigl(-25((x-0.5)^2+(y-0.75)^2)\bigr)\\
        &\quad + \exp\bigl(-30((x-0.8)^2+(y-0.35)^2)\bigr)
        \Bigr)\\
        \rho_I(0,x) &= 10 \bigl(0.04 - (x-0.2)^2-(y-0.65)^2 \bigr)_+\\
        &\quad + 12 \bigl(0.03 - (x-0.5)^2-(y-0.2)^2 \bigr)_+\\
        &\quad + 12 \bigl(0.03 - (x-0.8)^2-(y-0.55)^2 \bigr)_+\\
        \rho_R(0,x) &= 0.
    \end{aligned}
\end{equation*}
Susceptible population is the sum of three Gaussian distributions and infected population is the sum of positive part of quadratic polynomials. We conduct this experiment to show that the algorithm works well for nonsymmetric initial densities. Moreover, we choose $\beta = 0.34$ (an infection rate) and $\gamma = 0.12$ (a recovery rate) from~\cite{bertozzi2020challenges} based on the data in California, U.S. from March to May 2020. Figure~\ref{fig:exp3-1} shows the evolution of densities using these parameters. We repeat the experiment with same initial densities and $\beta$ but with different $\gamma$ (Figure~\ref{fig:exp3-2}). In this experiment, we show the solution of the problem based on a large $\gamma = 0.36$. This experiment is under the scenario when vaccine comes to the public. In both figures, evolutions of densities $\rho_i$ ($i\in \{S,I,R\}$) are shown at $t=0, 0.21, 0.47, 0.74, 1$. The total population of each density is indicated as \textit{sum} in the subtitle of each plot, and it is calculated as $\int_\Omega \rho_i(t,x) dx$ for $0\leq t \leq T$ and $i \in \{S,I,R\}$.

When $\gamma=0.12$ (a low recovery rate), the solution separates susceptible population away from infected population. By separating susceptible from infected, the solution prevents susceptible population becoming infected, thus reduces the terminal cost at $t=1$. When $\gamma=0.36$ (a high recovery rate), recovering the infected is considered to be a better choice than separating susceptible population from infected population. In Figure ~\ref{fig:exp3-2}, the susceptible population barely moves over time. We also observe that less number of infected and more number of recovered. The total population of infected at the terminal time in Figure~\ref{fig:exp3-2} is $0.045$ which is smaller than the total population of infected in Figure~\ref{fig:exp3-1}. This experiment tells us that, with a high recovery rate, the optimal way of minimizing the number of infected is by focusing on recovering them rather than moving susceptible population.

\begin{figure}[ht]
\includegraphics[width=1\linewidth]{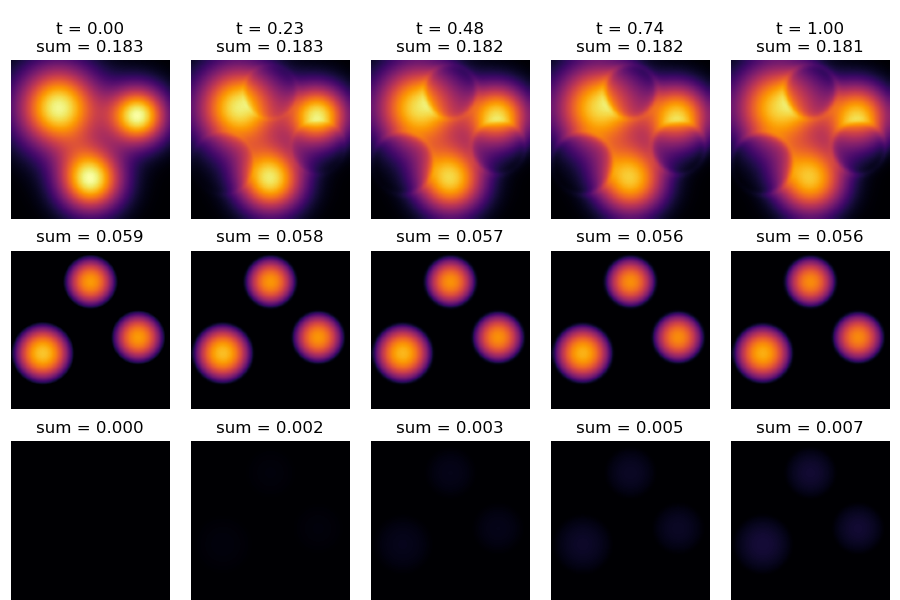}
\caption{Experiment 2. The evolution of populations from $t=0$ to $t=1$ with $\beta = 0.34$ and $\gamma = 0.12$. The first row represents susceptible, the second row represents infected, and the last row represents recovered.}
\label{fig:exp3-1}
\end{figure}

\begin{figure}[ht]
\includegraphics[width=1\linewidth]{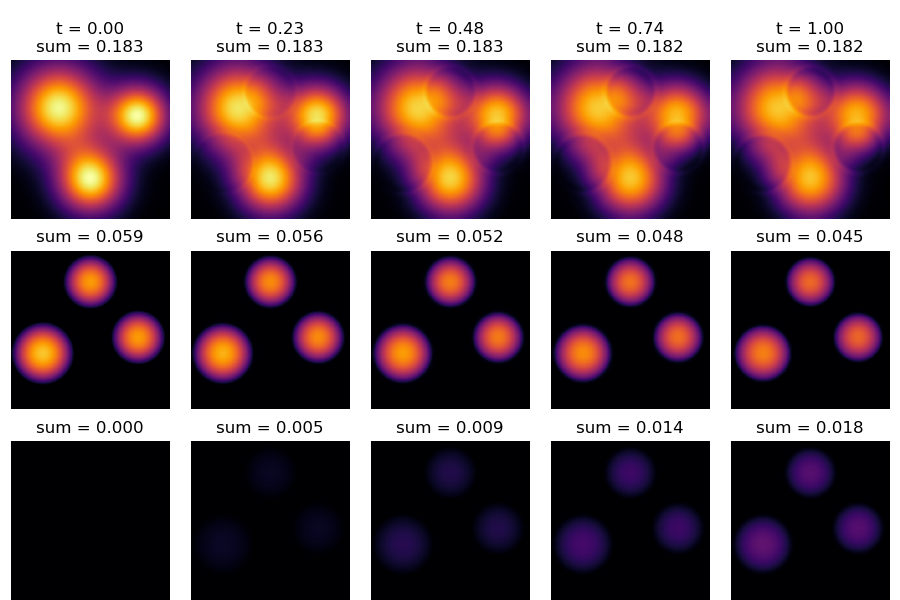}
\caption{Experiment 2. The evolution of populations from $t=0$ to $t=1$ with $\beta = 0.34$ and $\gamma = 0.36$. The first row represents susceptible, the second row represents infected, and the last row represents recovered.}
\label{fig:exp3-2}
\end{figure}

\subsection{Experiment 3}

In this experiment, we consider the initial densities
\begin{equation*}
    \begin{aligned}
        \rho_S(0,x) = \begin{cases}
            0.4 &\text{ if } x\in B_{0.3}(0.5,0.5)\\
            0 &\text{ else}
        \end{cases},\quad
        \rho_I(0,x) = \begin{cases}
            0.4  &\text{ if } x \in B_{0.2}(0.5,0.5)\\
            0 &\text{else}
        \end{cases} ,\quad \rho_R(0,\cdot) = 0
    \end{aligned}
\end{equation*}
where $B_R(x_1,x_2)$ is a ball of radius $R$ centered at $(x_1,x_2)$ with value. Furthermore, we consider the following energy functional:
\[
    E(\rho_I(T,\cdot)) = \int_\Omega \frac{1}{2} \rho_I^2(T,x) + \rho_I(T,x) V(x) dx
\]
where, for $x=(x_1,x_2)$,
\begin{equation*}
    V(x) =
    \begin{cases}
        1 & \text{if $|x_1-0.5|<0.1$ and $|x_2-0.5|<0.1$}\\
        0 & \text{otherwise.}
    \end{cases}
\end{equation*}
Here $V(x)$ is a step function that equals $1$ on a square with a side length $0.2$ at the center of the domain and $0$ elsewhere. This energy penalizes if there is positive infected density on the square. Thus, the solution has to move the infected density away from the square region while minimizing the total infected population. In this set of experiments, we show how the solution changes based on an infection rate $\beta$. We consider the case with a high infection rate $\beta=0.96, \gamma=0.12$ (Figure~\ref{fig:exp4-1}) and with $\beta=0.34, \gamma=0.12$ (Figure~\ref{fig:exp4-2}) same as Experiment~2.

When $\beta=0.96$ (a high infection rate), the solution minimizes the total infected population by separating the susceptible from infected. Due to the usage of this energy functional, the infected population has to move away from the square region at the center. Since there is going to be no infected population in this square region at the terminal time, the optimal place for susceptible population is inside this square region. As a result, we can see the concentrated susceptible population inside this square at the terminal time. When $\beta=0.32$ (a low infection rate), the susceptible population does not move as much as in the case when $\beta$ is large. There are more overlaps between susceptible and infected groups at the terminal time when $\beta$ is small. However, when $\beta$ is large, there is a complete separation between these groups. Thus, based on $\beta$ and $\gamma$ values, the solution of our model can find the most cost-effective way of moving susceptible and infected populations while minimizing the total infected population.

\begin{figure}[ht]
\includegraphics[width=1\linewidth]{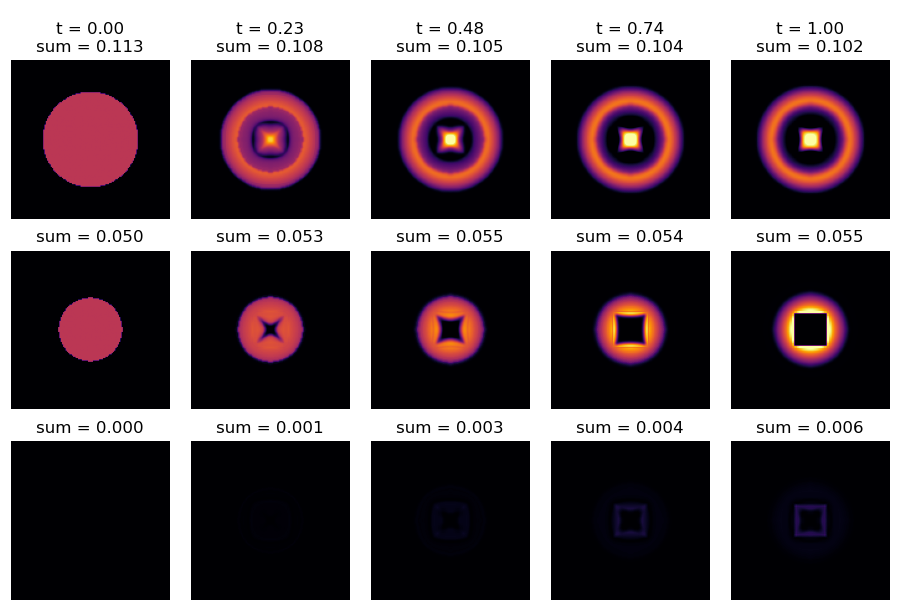}
\caption{Experiment 3. The evolution of populations from $t=0$ to $t=1$ with $\beta = 0.96$ and $\gamma = 0.12$. The first row represents susceptible, the second row represents infected, and the last row represents recovered.}
\label{fig:exp4-1}
\end{figure}

\begin{figure}[ht]
\includegraphics[width=1\linewidth]{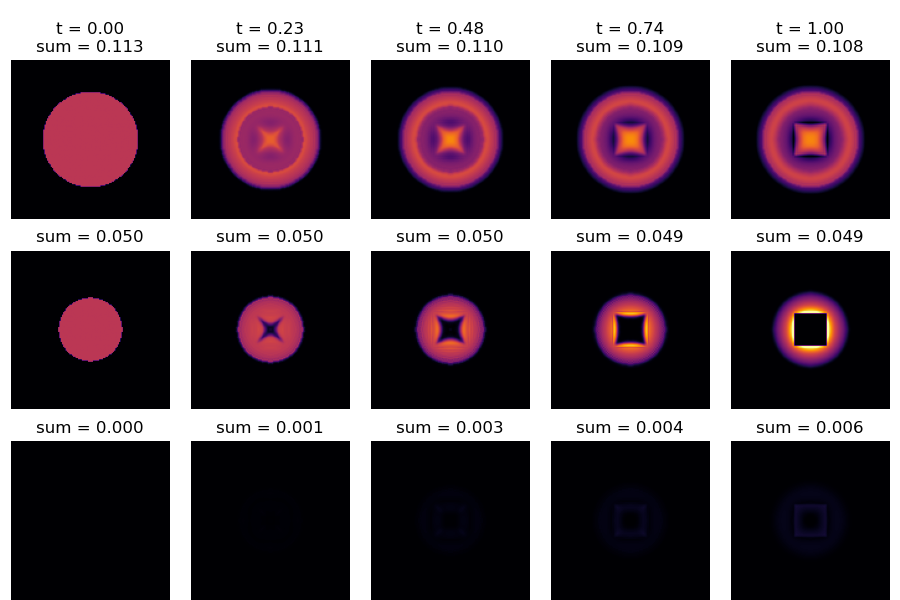}
\caption{Experiment 3. The evolution of populations from $t=0$ to $t=1$ with $\beta = 0.34$ and $\gamma = 0.12$. The first row represents susceptible, the second row represents infected, and the last row represents recovered.}
\label{fig:exp4-2}
\end{figure}

\section{Discussion}
In this paper, we introduce a mean-field control model for controlling the virus spreading of a population in a spatial domain, which extends and controls the current SIR model with spatial effect. Here the state variable represents the population status, such as S, I, R, etc with a spatial domain, while the control variable is the velocity of motion of the population. The terminal cost forms the goal of government, which balances the total infection number and maintain suitable physical movement of essential tasks and goods. Numerical algorithms are derived to solve the proposed model. Several experiments demonstrate that our model can effectively demonstrate how to separate the infected and susceptible population in a spatial domain. 

Our model opens the door to many questions in modeling, inverse problems and computations, especially during this COVID-19 pandemic. On the modeling side, first, we are interested in generalize the geometry of the spatial domain. Second, our current model only focuses on the control of population movement. The control of the diffusion operator among populations is also of great interests in future work. Third, the government can also put restrictions on the interaction for different class of populations, depending on their infection status. Fourth, in real life, the spatial domain is often inhomogeneous, containing airports, schools, subways etc. We also need to formulate our mean-field control model on a discrete spatial graph (network). In addition, our model focus on the forward problem of modeling the dynamics of the virus. In practice, real time data is generated as the virus spreading across different regions. To effectively model this dynamic, a suitable inverse mean-field control problem needs to be constructed. On the computational side, our model involves a non-convex optimization problem, which comes from the multiplicative term of the SIR model itself. In future work, we expect to design a fast and reliable algorithm for these advanced models. We also expect to develop and apply AI numerical algorithms to compute models in high dimensions.


\begin{thebibliography}{10}

\bibitem{kermack1927contribution}
William~Ogilvy Kermack and Anderson~G McKendrick.
\newblock A contribution to the mathematical theory of epidemics.
\newblock {\em Proceedings of the royal society of london. Series A, Containing
  papers of a mathematical and physical character}, 115(772):700--721, 1927.

\bibitem{kendall1965mathematical}
David~G Kendall.
\newblock Mathematical models of the spread of infection.
\newblock {\em Mathematics and computer science in biology and medicine}, pages
  213--225, 1965.

\bibitem{kallen1984thresholds}
Anders K{\"a}ll{\'e}n.
\newblock Thresholds and travelling waves in an epidemic model for rabies.
\newblock {\em Nonlinear Analysis: Theory, Methods \& Applications},
  8(8):851--856, 1984.

\bibitem{hosono1995traveling}
Yuzo Hosono and Bilal Ilyas.
\newblock Traveling waves for a simple diffusive epidemic model.
\newblock {\em Mathematical Models and Methods in Applied Sciences},
  5(07):935--966, 1995.

\bibitem{Jovanovic88}
Boyan Jovanovic and Robert~W. Rosenthal.
\newblock Anonymous sequential games.
\newblock {\em Journal of Mathematical Economics}, 17(1):77 -- 87, 1988.

\bibitem{HCM06}
M.~Huang, R.~P. Malham\'e, and P.~E. Caines.
\newblock Large population stochastic dynamic games: closed-loop
  {M}c{K}ean-{V}lasov systems and the {N}ash certainty equivalence principle.
\newblock {\em Commun. Inf. Syst.}, 6(3):221--251, 2006.

\bibitem{LasryLions06a}
Jean-Michel Lasry and Pierre-Louis Lions.
\newblock Jeux \`a champ moyen. {I}. {L}e cas stationnaire.
\newblock {\em C. R. Math. Acad. Sci. Paris}, 343(9):619--625, 2006.

\bibitem{LasryLions06b}
Jean-Michel Lasry and Pierre-Louis Lions.
\newblock Jeux \`a champ moyen. {II}. {H}orizon fini et contr\^{o}le optimal.
\newblock {\em C. R. Math. Acad. Sci. Paris}, 343(10):679--684, 2006.

\bibitem{lasry2007mean}
Jean-Michel Lasry and Pierre-Louis Lions.
\newblock Mean field games.
\newblock {\em Japanese journal of mathematics}, 2(1):229--260, 2007.

\bibitem{gomes2014mean}
Diogo~A Gomes et~al.
\newblock Mean field games models—a brief survey.
\newblock {\em Dynamic Games and Applications}, 4(2):110--154, 2014.

\bibitem{Gomes:2015th}
Diogo~A Gomes, Levon Nurbekyan, and Edgard~A Pimentel.
\newblock {\em Economic models and mean-field games theory}.
\newblock IMPA Mathematical Publications. Instituto Nacional de Matem\'atica
  Pura e Aplicada (IMPA), Rio de Janeiro, 2015.

\bibitem{burger2013mean}
Martin Burger, Marco Di~Francesco, Peter Markowich, and Marie-Therese Wolfram.
\newblock Mean field games with nonlinear mobilities in pedestrian dynamics.
\newblock {\em arXiv preprint arXiv:1304.5201}, 2013.

\bibitem{lachapelle2016efficiency}
Aim{\'e} Lachapelle, Jean-Michel Lasry, Charles-Albert Lehalle, and
  Pierre-Louis Lions.
\newblock Efficiency of the price formation process in presence of high
  frequency participants: a mean field game analysis.
\newblock {\em Mathematics and Financial Economics}, 10(3):223--262, 2016.

\bibitem{achdou19}
Yves Achdou and Jean-Michel Lasry.
\newblock Mean field games for modeling crowd motion.
\newblock In {\em Contributions to partial differential equations and
  applications}, volume~47 of {\em Comput. Methods Appl. Sci.}, pages 17--42.
  Springer, Cham, 2019.

\bibitem{bencar'15}
J.-D. Benamou and G.~Carlier.
\newblock Augmented {L}agrangian methods for transport optimization, mean field
  games and degenerate elliptic equations.
\newblock {\em J. Optim. Theory Appl.}, 167(1):1--26, 2015.

\bibitem{silva18}
L.~M. Brice\~{n}o Arias, D.~Kalise, and F.~J. Silva.
\newblock Proximal methods for stationary mean field games with local
  couplings.
\newblock {\em SIAM J. Control Optim.}, 56(2):801--836, 2018.

\bibitem{EHanLi2018_meanfield}
Weinan E, Jiequn Han, and Qianxiao Li.
\newblock A {{Mean}}-{{Field Optimal Control Formulation}} of {{Deep
  Learning}}.
\newblock {\em arXiv:1807.01083 [cs, math]}, 2018.

\bibitem{lin2020apacnet}
Alex~Tong Lin, Samy~Wu Fung, Wuchen Li, Levon Nurbekyan, and Stanley~J. Osher.
\newblock Apac-net: Alternating the population and agent control via two neural
  networks to solve high-dimensional stochastic mean field games, 2020.

\bibitem{ruthotto2019machine}
Lars Ruthotto, Stanley Osher, Wuchen Li, Levon Nurbekyan, and Samy~Wu Fung.
\newblock A machine learning framework for solving high-dimensional mean field
  game and mean field control problems, 2019.

\bibitem{liu2020computational}
Siting Liu, Matthew Jacobs, Wuchen Li, Levon Nurbekyan, and Stanley~J Osher.
\newblock Computational methods for nonlocal mean field games with
  applications.
\newblock {\em arXiv preprint arXiv:2004.12210}, 2020.

\bibitem{JacobsLegerLiOsher2018_solvinga}
Matt Jacobs, Flavien L\'eger, Wuchen Li, and Stanley Osher.
\newblock Solving {{Large}}-{{Scale Optimization Problems}} with a
  {{Convergence Rate Independent}} of {{Grid Size}}.
\newblock {\em arXiv:1805.09453 [math]}, 2018.

\bibitem{aronson1977asymptotic}
DG~Aronson.
\newblock The asymptotic speed of propagation of a simple epidemic.
\newblock In {\em Nonlinear diffusion}, volume~14, pages 1--23. Pitman London,
  1977.

\bibitem{diekmann1979run}
Odo Diekmann.
\newblock Run for your life. a note on the asymptotic speed of propagation of
  an epidemic.
\newblock {\em Journal of Differential Equations}, 33(1):58--73, 1979.

\bibitem{thieme1977model}
HR~Thieme.
\newblock A model for the spatial spread of an epidemic.
\newblock {\em Journal of Mathematical Biology}, 4(4):337--351, 1977.

\bibitem{caraco2002stage}
Thomas Caraco, Stephan Glavanakov, Gang Chen, Joseph~E Flaherty, Toshiro~K
  Ohsumi, and Boleslaw~K Szymanski.
\newblock Stage-structured infection transmission and a spatial epidemic: a
  model for lyme disease.
\newblock {\em The American Naturalist}, 160(3):348--359, 2002.

\bibitem{grenfell2001travelling}
Bryan~T Grenfell, Ottar~N Bj{\'o}rnstad, and Jens Kappey.
\newblock Travelling waves and spatial hierarchies in measles epidemics.
\newblock {\em Nature}, 414(6865):716--723, 2001.

\bibitem{wang2010travelling}
Zhi-Cheng Wang and Jianhong Wu.
\newblock Travelling waves of a diffusive kermack--mckendrick epidemic model
  with non-local delayed transmission.
\newblock {\em Proceedings of the Royal Society A: Mathematical, Physical and
  Engineering Sciences}, 466(2113):237--261, 2010.

\bibitem{berestycki2020propagation}
Henri Berestycki, Jean-Michel Roquejoffre, and Luca Rossi.
\newblock Propagation of epidemics along lines with fast diffusion.
\newblock {\em arXiv preprint arXiv:2005.01859}, 2020.

\bibitem{murray2001mathematical}
JD~Murray.
\newblock {\em Mathematical biology II: spatial models and biomedical
  applications}.
\newblock Springer New York, 2001.

\bibitem{ruan2007spatial}
Shigui Ruan.
\newblock Spatial-temporal dynamics in nonlocal epidemiological models.
\newblock In {\em Mathematics for life science and medicine}, pages 97--122.
  Springer, 2007.

\bibitem{chinviriyasit2010numerical}
Settapat Chinviriyasit and Wirawan Chinviriyasit.
\newblock Numerical modelling of an sir epidemic model with diffusion.
\newblock {\em Applied Mathematics and Computation}, 216(2):395--409, 2010.

\bibitem{jaichuang2014numerical}
Atit Jaichuang and Wirawan Chinviriyasit.
\newblock Numerical modelling of influenza model with diffusion.
\newblock {\em International Journal of Applied Physics and Mathematics},
  4(1):15, 2014.

\bibitem{farago2016qualitatively}
Istv{\'a}n Farag{\'o} and R{\'o}bert Horv{\'a}th.
\newblock Qualitatively adequate numerical modelling of spatial sirs-type
  disease propagation.
\newblock {\em Electronic Journal of Qualitative Theory of Differential
  Equations}, 2016(12):1--14, 2016.

\bibitem{sethi1978optimal}
Suresh~P Sethi and Preston~W Staats.
\newblock Optimal control of some simple deterministic epidemic models.
\newblock {\em Journal of the Operational Research Society}, 29(2):129--136,
  1978.

\bibitem{lahrouz2018dynamics}
A~Lahrouz, H~El~Mahjour, A~Settati, and A~Bernoussi.
\newblock Dynamics and optimal control of a non-linear epidemic model with
  relapse and cure.
\newblock {\em Physica A: Statistical Mechanics and its Applications},
  496:299--317, 2018.

\bibitem{jang2020optimal}
Junyoung Jang, Hee-Dae Kwon, and Jeehyun Lee.
\newblock Optimal control problem of an sir reaction--diffusion model with
  inequality constraints.
\newblock {\em Mathematics and Computers in Simulation}, 171:136--151, 2020.

\bibitem{li2019dynamic}
Kezan Li, Guanghu Zhu, Zhongjun Ma, and Lijuan Chen.
\newblock Dynamic stability of an siqs epidemic network and its optimal
  control.
\newblock {\em Communications in Nonlinear Science and Numerical Simulation},
  66:84--95, 2019.

\bibitem{allen2017primer}
Linda~JS Allen.
\newblock A primer on stochastic epidemic models: Formulation, numerical
  simulation, and analysis.
\newblock {\em Infectious Disease Modelling}, 2(2):128--142, 2017.

\bibitem{champock11}
Antonin Chambolle and Thomas Pock.
\newblock A first-order primal-dual algorithm for convex problems with
  applications to imaging.
\newblock {\em J. Math. Imaging Vision}, 40(1):120--145, 2011.

\bibitem{champock16}
Antonin Chambolle and Thomas Pock.
\newblock On the ergodic convergence rates of a first-order primal-dual
  algorithm.
\newblock {\em Math. Program.}, 159(1-2, Ser. A):253--287, 2016.

\bibitem{bertozzi2020challenges}
Andrea~L Bertozzi, Elisa Franco, George Mohler, Martin~B Short, and Daniel
  Sledge.
\newblock The challenges of modeling and forecasting the spread of covid-19.
\newblock {\em arXiv preprint arXiv:2004.04741}, 2020.

\end{thebibliography}
\end{document}